\documentclass[sigconf]{acmart}
\renewcommand\footnotetextcopyrightpermission[1]{}
\usepackage{booktabs}
\usepackage{algorithmic}
\usepackage{algorithm}

\usepackage[justification=centering]{caption}
\usepackage{subcaption}

\usepackage{enumitem}
\setlist{noitemsep,topsep=0pt,parsep=0pt,partopsep=0pt,leftmargin=*}

\usepackage{multicol}

\newtheorem{remark}{Remark}[section]

\theoremstyle{remark}

\newcommand{\bbR}{\mathbb{R}}

\newcommand{\bbE}{\mathbb{E}}
\newcommand{\bbP}{\mathbb{P}}
\newcommand{\calS}{\mathcal{S}}

\newcommand{\bei}{\begin{itemize}}
\newcommand{\eei}{\end{itemize}}
\newcommand{\bee}{\begin{enumerate}}
\newcommand{\eee}{\end{enumerate}}

\newcommand{\calI}{\mathcal{I}}
\newcommand{\calK}{\mathcal{K}}
\newcommand{\calE}{\mathcal{E}}
\newcommand{\bx}{\mathbf{x}}
\newcommand{\bb}{\mathbf{b}}

\newcommand{\bud}{m}
\newcommand{\INDSTATE}[1][1]{\STATE\hspace{#1\algorithmicindent}}

\setcopyright{none}

\acmDOI{10.475/123_4}

\begin{document}

\title{Profit Maximization for Online Advertising Demand-Side Platforms}

\author{Paul Grigas, Alfonso Lobos}
\affiliation{
  \institution{University of California, Berkeley}
}
\email{pgrigas, alobos@berkeley.edu}

\author{Zheng Wen}
\affiliation{
  \institution{Adobe Research}
}
\email{zwen@adobe.com}

\author{Kuang-chih Lee}
\affiliation{
  \institution{Yahoo Inc.}
}  
\email{kclee@yahoo-inc.com}

\begin{abstract}
We develop an optimization model and corresponding algorithm for the management of a demand-side platform (DSP), whereby the DSP aims to maximize its own profit while acquiring valuable impressions for its advertiser clients. We formulate the problem of profit maximization for a DSP interacting with ad exchanges in a real-time bidding environment in a cost-per-click/cost-per-action pricing model. Our proposed formulation leads to a nonconvex optimization problem due to the joint optimization over both impression allocation and bid price decisions. We use Lagrangian relaxation to develop a tractable convex dual problem, which, due to the properties of second-price auctions, may be solved efficiently with subgradient methods. We propose a two-phase solution procedure, whereby in the first phase we solve the convex dual problem using a subgradient algorithm, and in the second phase we use the previously computed dual solution to set bid prices and then solve a linear optimization problem to obtain the allocation probability variables. On several synthetic examples, we demonstrate that our proposed solution approach leads to superior performance over a baseline method that is used in practice.
\end{abstract}

\maketitle

\section{Introduction}\label{sec:intro}

In targeted online advertising, the main goal is to figure out the best opportunities by showing an advertisement to an online user, who is most likely to take a desired action, such as ordering a product or signing up for an account. The complexity of realizing this goal is so high that advertisers need specialized technology solutions called demand-side platforms (DSP). 

In a DSP, each individual advertiser usually sets up a list of campaigns that can be thought of as plans for delivering advertisements. For each campaign, the advertiser specifies the characteristics of the audience segments that she would like to target (e.g., males, ages 18-35, who view news articles on espn.com) along with the particular media that she would like to display to the target audience (e.g., a video ad for beer). In addition, the advertiser specifies a budget amount, time schedule, pacing details, and performance goals for each campaign. The performance goals typically can be specified by minimizing cost-per-click (CPC) or cost-per-action (CPA). 

The DSP manages its active campaigns for many different advertisers simultaneously across multiple ad exchanges where ad impressions can be acquired through a real-time bidding (RTB) process. In the RTB process, the DSP interacts with several ad exchanges where bids are placed for potential impressions on behalf of those advertisers. This interaction happens in real time when an ad request is submitted to an ad exchange (which may happen, for example, when a user views a news story on a webpage). In this scenario, the DSP needs to offer a solution to decide, among the list of all campaigns associated with its advertiser clients, which campaign to bid on behalf of and how much to offer for the corresponding bid. The fundamental problem we consider here is how to make these decisions in real time to maximize the profit for the DSP while ensuring that all of its advertiser clients are satisfied.

Generally speaking, today most DSPs offer different pricing models and enhancement schemes to help advertisers manage their campaigns. Those pricing models include cost per thousand impressions (CPM), cost per click (CPC) and cost per action (CPA). Advertisers often like to choose the CPC/CPA pricing model because the return on investment (ROI) is always positive. However, this pricing model might introduce revenue loss for the DSP since the DSP only earns revenue when a click or action occurs. Therefore, in CPC/CPA pricing model, the DSP needs to convert the CPC/CPA bid to an expected cost per thousand impressions (eCPM) bid in order to sensibly participate in auctions in the RTB exchanges. In this paper, we focus on CPC/CPA pricing model as it is a very challenging problem for the DSP. 

It is challenging for a DSP to perform such profit optimization with CPC/CPA pricing model in a RTB environment for several reasons. First, top DSPs typically receive as many as a million ad requests per second. The short latency and high throughput requirements introduce extreme time sensitivity on the process. Second, a large amount of information is missing in the real time evaluation of the individual ad requests, e.g., the feedback on previous decisions normally has a long delay in practice. Therefore, most of the DSPs today only apply a greedy approach by selecting the ad with the highest bid among all the qualified ads for each incoming request. 

In this paper, we propose a novel approach based on a precise mathematical formulation to optimize the overall DSP profit. We appropriately model the uncertainty in impression arrival, auction, and click/action processes and develop an optimization formulation to maximize profit for the DSP while ensuring that each campaign remains under budget. Our formulation is aimed at optimizing with respect to both impression allocation and bid price decisions, and due to the additional complexity of accounting for both of these decisions the formulation is a large-scale, nonconvex model. However, due to the properties of second-price auctions we are able to effectively use the technique of Lagrangian relaxation. We construct a dual problem and establish that subgradients of the dual function may be efficiently computed. Our overall approach is based on a two-phase procedure, wherein we solve the dual problem in the first phase and use the dual solution to naturally recover a primal solution in the second phase. We conduct several computational experiments on synthetic datasets and demonstrate that our Lagrangian relaxation based approach is able to significantly increase DSP profits relative to a baseline greedy approach.

Revenue optimization in online advertising has been extensively studied in recent literature from different perspectives, such as optimization \cite{mehta2007adwords, balseiro2014yield, chen2011real, zhang2014optimal}, game theory and mechanism design \cite{balseiro2015repeated, maehara2015budget}, and contract design \cite{mirrokni2017deals}. Due to space limitations, we only review several directly relevant papers. \cite{balseiro2014yield, chen2014dynamic, chen2011real} focus on the publisher's revenue management problem. Specifically, \cite{balseiro2014yield, chen2014dynamic} study how publishers should optimally trade off guaranteed contracts with RTB. \cite{chen2011real} studies how a publisher should optimally allocate impressions and set up bid prices for campaigns, under the implicit assumption that the publisher is a ``central planner''. On the other hand, \cite{ciocan2012model} studies ad networks' revenue management problem based on model predictive control. Finally, \cite{zhang2014optimal} studies advertisers' optimal bidding problem in RTB. Unlike our paper, \cite{zhang2014optimal} focuses on optimal bidding and its framework does not consider impression allocation.

The rest of the paper is organized as follows. 
In Section \ref{sec:model}, we describe the notation and problem statement and we set up the model. The profit optimization formulation is presented in Section \ref{sec:optalg}. Our proposed Lagrangian relaxation based algorithm to approximately solve the profit optimization problem is specified in Section \ref{sec:LAD}. Experimental results are presented in Section \ref{sec:expts}, and we conclude with some remarks about possible future work.

\section{Model Foundations}\label{sec:model}

We assume that the planning takes place over a fixed time horizon (e.g., 24 hours). To simplify the presentation, we also assume that the DSP interacts with a single ad exchange. For our purposes, the ad exchange simply represents a pool of potential impressions that the DSP may bid on. Thus, in the likely scenario that the DSP interacts with more than one ad exchange, we may simply group all of the different ad exchanges together into a representative ad exchange.

Let $\calI$ denote the set of \emph{impression types} and let $\calK$ denote the set of all campaigns associated with advertisers managed by the DSP. Before discussing the details of our assumptions, let us describe the basic flow of events in the model. When an impression of type $i \in \calI$ is submitted to the ad exchange, a real-time second-price auction is held for which the DSP has an opportunity to bid. Thus the DSP has an opportunity to make two strategic decisions related to each real-time auction:  \emph{(i)} how to select a campaign $k \in \calK$ to bid on behalf of in the auction, and \emph{(ii)} how to set the corresponding bid price $b_{ik}$. If the DSP wins the auction on behalf of campaign $k$, then the DSP must pay the ad exchange an amount equal to the second largest price and an ad from campaign $k$ is displayed. The advertiser corresponding to campaign $k$ is charged only if a user clicks on the ad.

\paragraph{Impression Types}
It is important that the set of impression types $\calI$ represents a partition of all possible impressions that are submitted to the ad exchange. Thus, every impression submitted to the ad exchange is associated with a particular impression type $i \in \calI$. It is most natural to define $\calI$ in terms of features associated with impressions. For example, if the DSP determines that there are only two relevant attributes associated with each impression -- say gender and whether or not the viewer is 18 years or older -- then the DSP would choose $\calI = \{(M, 18-), (M, 18+), (F, 18-), (F, 18+)\}$. Then, all impressions corresponding to male viewers who are under 18 years of age would be assigned to impression type $i = (M, 18-)$, etc. Note that the construction of the set $\calI$ is part of the modeling process and consideration should be given to the trade-off between computational limitations and the potential for higher profits due to a more fine-grained construction of $\calI$. Nevertheless, the algorithmic schemes we propose in Section \ref{sec:optalg} are scalable to problem instances where the size of $\calI$ is extremely large. We use the following notation and make the following assumptions about the impression types:
\begin{itemize}
\item Let $S_i$ denote the number of impressions of type $i$ that arrive during the planning horizon, and assume that $S_i$ is a random variable with mean $s_i$. 
\end{itemize}

\paragraph{Campaigns}
Recall that $\calK$ denotes the set of all campaigns that are managed by the DSP. That is, $\calK$ is the union over all advertisers (who are managed by the DSP) of the sets of campaigns run by each advertiser. We use the following notation and make the following assumptions about the campaigns:
\begin{itemize}
\item $\bud_k$ denotes the (advertiser selected) budget for campaign $k$ during the planning horizon.
\item $\calI_k$ denotes the set of impression types that campaign $k$ targets. For example, if an advertiser wishes to create a particular campaign to target female users, then in this case $\calI_k$ would denote the set of all impression types corresponding to female users (e.g., $\calI_k = \{(F, 18-), (F, 18+)\}$ in the example described above). 
Note that each advertiser can create multiple campaigns to achieve different targeting goals.
\item In this model, it is assumed that advertisers are charged on a \emph{CPC (cost per click)} basis. That is, campaign $k$ is charged an amount $q_k > 0$, called the CPC price, each time a user clicks on an advertisement from campaign $k$. (Note that a ``click'' may also be thought of more generally as an ``action'' whereby our model easily extends to campaigns that are charged on a CPA basis. Moreover, our model may be easily extended to allow for multiple actions, each with their own rewards.)
\end{itemize}

\paragraph{Auctions}
When an impression is submitted to the ad exchange, an instantaneous real-time auction occurs to determine who gets to display an advertisement. We assume that these are \emph{second-price} auctions, which are very common in practice. In a second-price auction, the bidder who submits the highest bid is the winner, but the amount that the winner pays is the amount of the \emph{second highest} bid. It is well known that, in a second-price auction, a dominant strategy for each participant is to bid truthfully \cite{vickrey1961counterspeculation}. 
Herein we assume that the DSP takes a \emph{probabilistic} approach to modeling the behavior of the other bidders in the auction. Namely, we make the following assumptions:
\begin{itemize}
\item For each impression type $i \in I$, let $B^{\max}_i$ be a random variable representing the maximum, among all other bidders excluding the DSP, of the bid prices entered in an auction for an impression of type $i$. It is assumed that $B^{\max}_i > 0$ with probability one. Let $\rho_i(\cdot) : \bbR \to [0,1]$ denote the cumulative distribution function (CDF) of $B^{\max}_i$, so that $\rho_i(b) := \bbP(B^{\max}_i \leq b)$ is the probability that the DSP wins an auction for an impression of type $i$ when the DSP enters the bid $b$. Note that the functions $\rho_i(\cdot)$ are typically estimated using a bid landscape model (see, e.g., \cite{cui2011bid}).
(In the case of a tie, here we assume that the DSP automatically wins the auction. Our framework may be easily adapted to the case of a fair tie-breaking mechanism.)
\item Furthermore, let $\beta^{\max}_i(b) := \bbE[B^{\max}_i ~|~ B^{\max}_i \leq b]$ denote the expected value of the second highest bid price (i.e., the value of the payment to the ad exchange) given that the DSP enters a bid price of $b$ and $b$ is the largest bid price entered.
\end{itemize}

\paragraph{Click Events}
After the DSP has won an auction on behalf of campaign $k$, an ad for campaign $k$ is displayed to the user corresponding to the impression for which the auction was held. For a given impression type $i \in \calI$ and a given campaign $k \in \calK$, let $\theta_{ik} \in [0,1]$ denote the \emph{click-through-rate} for users corresponding to impression type $i$ and when the ad corresponds to campaign $k$. That is, $\theta_{ik}$ represents the fraction of users corresponding to impression type $i$ that click on an ad associated with campaign $k$, i.e., the probability that the user clicks on the ad that is shown. 
Although the true click-through-rates are not available, the DSP is typically able to leverage a vast amount of historical data and use predictive models to produce accurate predictions of these values, even when $\calK$ and $\calI$ are extremely large (see, e.g., \cite{mcmahan2013ad}). 

Finally, given an impression type $i \in \calI$ and a campaign $k \in \calK$, let $r_{ik}$ denote the \emph{expected cost per impression (eCPI)} value, namely $r_{ik} := q_k\theta_{ik}$ where $q_k$ is the CPC price defined earlier. Note that $r_{ik}$ is the expected amount of revenue that the DSP earns each time an ad for campaign $k$ is shown to an impression of type $i$, and $r_{ik}$ also corresponds to the optimal bid price when campaign $k$ has unlimited budget.

\paragraph{Decision Variables and Additional Notation}
As mentioned previously, when an auction for impression type $i \in \calI$ arrives to the ad exchange, the DSP decides which campaign $k \in \calK$ to bid of behalf of and also selects the value of the corresponding bid price. Let $\calE \subseteq \calI \times \calK$ denote the edges of an undirected bipartite graph between $\calI$ and $\calK$, whereby there is an edge $e = (i, k) \in \calE$ whenever campaign $k$ targets impression type $i$, i.e, $\calE := \{(i, k) : i \in \calI_k\}$. Let $\calK_i := \{k \in \calK : (i, k) \in \calE\}$ denote the set of campaigns that target impression type $i$.

When a new auction for impression type $i$ arrives to the ad exchange, we say that the DSP selects campaign $k$ for the auction if the DSP chooses to bid on behalf campaign $k$ in the auction. For each edge $(i,k) \in \calE$, we define two decision variables as follows:  {\em (i)} $x_{ik}$ is the probability that the DSP selects campaign $k$, and {\em (ii)} $b_{ik}$ is the corresponding bid price that the DSP submits in the auction. Interpreted differently, $x_{ik}$ represents a proportional allocation, i.e., the fraction of auctions for impression type $i$ that are allocated to campaign $k$ on average. Note that $b_{ik}$ represents the bid price that the DSP submits to an auction for impression type $i$ \emph{conditional} on the fact that the DSP has selected campaign $k$ for the auction. Related approaches (e.g., as in \cite{chen2011real}) also use bid prices to rank advertisers -- in our approach, the selection of which campaign to bid on behalf of is completely captured by the $x_{ik}$ decision variables and thus the $b_{ik}$ decision variables only determine the actual bid price decisions. Let $\bf{x}, \bf{b} \in \bbR^{|\calE|}$ denote vectors of these quantities, which will represent decision variables in our model.

Let us also define some additional notation used herein. For a given set $S$ and a function $f(\cdot) : S \to \bbR$, let $\arg\max_{x \in S} f(x)$ denote the (possibly empty) set of maximizers of the function $f(\cdot)$ over the set $S$. If $f(\cdot) : \bbR^n \to \bbR$ is a convex function then, for a given $x \in \bbR^n$, $\partial f(x)$ denotes the set of subgradients of $f(\cdot)$ at $x$, i.e., the set of vectors $g$ such that $f(y) \geq f(x) + g^T(y - x)$ for all $y \in \bbR^n$. Finally, let $\mathbf{1}(\cdot)$ denote an indicator function that is equal to $1$ whenever the argument of $\mathbf{1}(\cdot)$ is true and equal to $0$ otherwise. 

\section{Optimization Formulation}
\label{sec:optalg}
Problem \eqref{poi_deterministic} presents our formulation of the allocation and real-time bidding planning problem faced by the DSP.
\begin{equation}\label{poi_deterministic}
\begin{array}{lcl}
\underset{\mathbf{x}, \mathbf{b}}{\mathrm{maximize}} & & \displaystyle \sum_{(i,k) \in \calE} [r_{ik} - \beta^{\max}_i(b_{ik})]s_ix_{ik}\rho_i(b_{ik}) \\
& & \\
\vspace{0.1cm}
\mathrm{subject \ to} & & \textstyle \sum_{i \in \calI_k}r_{ik}s_ix_{ik}\rho_i(b_{ik}) ~\leq~ m_k \ \  \forall  k \in \calK \\
\vspace{0.1cm}
& & \textstyle \sum_{k \in \calK_i} x_{ik} ~\leq~ 1 \ \ \forall i \in \calI \\
& & \mathbf{x}, \mathbf{b} ~\geq~ 0 \ .
\end{array}
\end{equation}
Herein, let $\pi(\bx, \bb) := \sum_{(i,k) \in \calE} [r_{ik} - \beta^{\max}_i(b_{ik})]s_ix_{ik}\rho_i(b_{ik})$ denote the objective function of \eqref{poi_deterministic}. Let us now briefly describe the interpretation of each part of the formulation \eqref{poi_deterministic} above. First, note that the formulation is based on the idea of ``deterministic approximation,'' whereby we assume that all random quantities deterministically take on their expected values. In this formulation, the DSP seeks to maximize its total profit over the planning horizon, while ensuring that each campaign does not spend more than its budget. Indeed, the expected number of times during the planning horizon that the DSP selects campaign $k$ for an auction of impression type $i$ is $s_ix_{ik}$ and the expected number of such auctions that the DSP wins is $s_ix_{ik}\rho_i(b_{ik})$. Furthermore, for each instance that the DSP selects campaign $k$ and wins the corresponding auction for impression type $i$, the expected profit for the DSP is $r_{ik} - \beta^{\max}_i(b_{ik})$. Therefore, the objective function of \eqref{poi_deterministic} represents the expected total profit earned by the DSP throughout the planning horizon. The first set of constraints in \eqref{poi_deterministic} represent the \emph{budget constraints} for the campaigns, which ensure that, in expectation, each campaign does not spend more than its pre-specified budget level. Finally, the second set of constraints in \eqref{poi_deterministic} are referred to as the \emph{supply constraints} for the impression types, which (along with the nonnegativity constraints on $\mathbf{x}$) ensure that the variables $\mathbf{x}$ represent valid probabilities. Note also that these probabilities may sum to a value strictly less than 1, in which case $1 - \sum_{k \in \calK_i} x_{ik}$ represents the probability of electing not to bid when an impression of type $i$ arrives to the ad exchange. For ease of notation, let us denote the supply and nonnegativity constraints on $\mathbf{x}$ using $\calS := \left\{\mathbf{x} : \sum_{k \in \calK_i} x_{ik} ~\leq~ 1 \ \text{for all } i \in \calI, \ \mathbf{x} \geq 0\right\}$.

Note that, due to the joint optimization over both $\bf{x}$ and $\bf{b}$, problem \eqref{poi_deterministic} is generally a nonconvex optimization problem (this is clearly evident, for example, when $B^{\max}_i$ is taken to be uniformly distributed on $[0,1]$ and hence $\rho_i(b_{ik}) = b_{ik}$ for $b_{ik} \in [0,1]$). Despite its nonconvexity, problem \eqref{poi_deterministic} does have some important structural properties that we now highlight. First, if we consider the bid prices $\mathbf{b}$ to be fixed, then the resulting problem in $\mathbf{x}$ is a \emph{linear optimization} problem -- in other words the objective function and constraints can all be expressed as linear functions of $\mathbf{x}$ -- and may be solved very efficiently using off-the-shelf solvers or perhaps a specialized algorithm. Conversely, if we consider $\mathbf{x}$ to be fixed, then the resulting problem in $\mathbf{b}$ is generally still nonconvex but the main ``difficulty'' arises from the budget constraints. Indeed, due to the presence of budget constraints, it may be optimal for the DSP to underbid on a relatively less valuable impression due to the possibility of a more valuable impression arriving in the future. Therefore, whenever a campaign has unlimited budget, it is optimal for the DSP to set $b_{ik} = r_{ik}$, i.e., to bid truthfully. The following Proposition, which will be useful in the development of the Lagrangian relaxation algorithm in Section \ref{sec:LAD}, formalizes this intuition.

\begin{proposition}\label{prop:truth}
Consider the following modification of \eqref{poi_deterministic} without budget constraints:
\begin{equation}\label{poi_prop}
\begin{array}{lcl}
\underset{\mathbf{x}, \mathbf{b}}{\mathrm{maximize}} & & \displaystyle \sum_{(i,k) \in \calE} [r_{ik} - \beta^{\max}_i(b_{ik})]s_ix_{ik}\rho_i(b_{ik}) \\
& & \\
\mathrm{subject \ to} & & \mathbf{x} \in \calS \qquad
\textrm{  and  } \qquad \mathbf{b} ~\geq~ 0 \ ,
\end{array}
\end{equation}
where $\calS := \left\{\mathbf{x} : \sum_{k \in \calK_i} x_{ik} ~\leq~ 1 \ \text{for all } i \in \calI, \ \mathbf{x} \geq 0\right\}$, as defined earlier.
Define $(\mathbf{x}^\ast, \mathbf{b}^\ast)$ by $b_{ik}^\ast := r_{ik}$ for all $(i,k) \in \calE$ and by letting $\bx^\ast$ be an arbitrary optimal solution of the resulting linear optimization problem, i.e.,
$\mathbf{x}^\ast \in \arg\max\limits_{\bx \in \calS}\left\{\sum_{(i,k) \in \calE}\pi_{ik}x_{ik}\right\} $,
where $\pi_{ik} := [r_{ik} - \beta^{\max}_i(r_{ik})]s_i\rho_i(r_{ik})$ for all $(i,k) \in \calE$.
Then, \eqref{poi_prop} is finite and $(\mathbf{x}^\ast, \mathbf{b}^\ast)$ is an optimal solution of \eqref{poi_prop}.
\end{proposition}
\begin{proof}
Fix an arbitrary $\mathbf{x} \in \calS$ and consider the resulting problem in the $\mathbf{b}$ variables:
\begin{equation}\label{poi_prop2}
\begin{array}{lcl}
\underset{\mathbf{b}}{\mathrm{maximize}} & & \displaystyle \sum_{(i,k) \in \calE} [r_{ik} - \beta^{\max}_i(b_{ik})]s_ix_{ik}\rho_i(b_{ik}) \\
\mathrm{subject \ to} & & \mathbf{b} ~\geq~ 0 \ .
\end{array}
\end{equation}
We now demonstrate that, regardless of the value of $\mathbf{x}$, an optimal solution of \eqref{poi_prop2} is given by $\bb^\ast$, i.e., by setting $b_{ik}^\ast = r_{ik}$ for all $(i,k) \in \calE$. Indeed, note that \eqref{poi_prop2} is completely separable across the $b_{ik}$ variables. Thus, since $x_{ik} \geq 0$, for each $(i,k) \in \calE$ we simply need to independently solve
\begin{equation}\label{poi_prop3}
\begin{array}{lcl}
\underset{b_{ik}}{\mathrm{maximize}} & & [r_{ik} - \beta^{\max}_i(b_{ik})]\rho_i(b_{ik})\mathbf{1}(x_{ik} > 0) \\
\mathrm{subject \ to} & & b_{ik} ~\geq~ 0 \ ,
\end{array}
\end{equation}
where $\mathbf{1}(x_{ik} > 0)$ is an indicator function that is equal to 1 whenever $x_{ik} > 0$ and 0 otherwise. If $x_{ik} = 0$, then the objective function in \eqref{poi_prop3} is always 0 and hence any value of $b_{ik} \geq 0$, in particular $b_{ik}^\ast := r_{ik}$, is an optimal solution. Otherwise, if $x_{ik} > 0$, then the objective function in \eqref{poi_prop3} is just the expected utility when entering a bid of $b_{ik}$ into a second price auction when the valuation is equal to $r_{ik} \geq 0$, for which the dominant strategy is to bid truthfully. Thus, in either case, it is clear that $b_{ik}^\ast := r_{ik}$ is an optimal solution of \eqref{poi_prop3} and it follows that the vector $\mathbf{b}^\ast$ is an optimal solution of \eqref{poi_prop2}. Recall that $\pi(\bx, \bb)$ denotes the objective function of \eqref{poi_deterministic} and hence also the objective function of \eqref{poi_prop}. Then we have shown that $\pi(\bx, \bb^\ast) \geq \pi(\bx, \bb)$ for all $\bb \geq 0$.
Thus, since $\bx \in \calS$ was selected arbitrarily, we have:
\begin{equation*}
\pi(\bx^\ast, \bb^\ast) \geq \pi(\bx, \bb^\ast) \geq \pi(\bx, \bb) \ \ \text{for all } \bx \in \calS, \bb \geq 0 \ ,
\end{equation*}
which shows that $(\mathbf{x}^\ast, \mathbf{b}^\ast)$ is an optimal solution of \eqref{poi_prop}.
\end{proof}

\begin{remark}\label{calScompute}
Given coefficients $\pi_{ik} \in \bbR$ for each $(i,k) \in \calE$, an optimal solution $\bx^\ast$ of the linear optimization problem
$\max\limits_{\bx \in \calS}\left\{\sum_{(i,k) \in \calE}\pi_{ik}x_{ik}\right\}$
may be computed efficiently in $O(|\calE|)$ time using a simple greedy algorithm. Namely, for each $i \in \calI$, we compute $k^\ast(i) \in \arg\max_{k \in \calK_i}\left\{\pi_{ik}\right\}$, set $x_{ik^\ast(i)} = \mathbf{1}(\pi_{ik^\ast(i)} > 0)$, and set $x_{ik} = 0$ for all other $k \in \calK_i$.
\end{remark}

\section{Lagrangian Dual and Algorithmic Scheme}\label{sec:LAD}
We begin this section with a high-level description of our approach for solving \eqref{poi_deterministic}. Our algorithmic approach is based on a two phase procedure. In the first phase, we construct a suitable dual of \eqref{poi_deterministic}, which turns out to be a convex optimization problem that can be efficiently solved with most subgradient based algorithms. A solution of the dual problem naturally suggests a way to set the bid prices $\mathbf{b}$. In the second phase, we set the bid prices using the previously computed dual solution then we solve the linear optimization problem that results when $\bb$ is fixed in order to recover allocation probabilities $\bx$. 

Let us now construct a Lagrangian dual of the deterministic approximation problem \eqref{poi_deterministic} by relaxing the ``difficult'' budget constraints. We show that the resulting dual problem is a convex optimization problem with only very simple box constraints and that subgradients of the objective function may be efficiently computed. Since our formulation is quite general (for example, there are no strong assumptions made about the distribution of $B^{\max}_i$), we are unable to exploit any special structure of the dual function $L^\ast(\cdot)$ and must resort to simple subgradient based algorithms to solve the dual problem. Nevertheless, subgradient methods offer the advantage of being highly scalable and parallelizable, and moreover our overall two phase procedure does not necessitate a high accuracy solution of the dual problem. 

To start, we introduce multipliers $\lambda \in \bbR^{|\calK|}_{+}$ for the budget constraints in \eqref{poi_deterministic} and form the Lagrangian function:
\begin{equation}\label{original_lag}
\begin{split}
L(\bx, \bb, \lambda) ~:=~ & \textstyle \sum_{(i,k) \in \calE}[r_{ik} - \beta^{\max}_i(b_{ik})]s_ix_{ik}\rho_i(b_{ik}) \\ 
+ & \textstyle \sum_{k \in \calK} \lambda_k \left[ m_k ~-~ \sum_{i \in \calI_k}r_{ik}s_ix_{ik}\rho_i(b_{ik}) \right] \ .
\end{split}
\end{equation}
After rearranging, we may re-express the Lagrangian function:
\begin{equation}\label{rearrange}
\begin{split}
L(\bx, \bb, \lambda) = & \textstyle \sum_{(i,k) \in \calE }[(1 - \lambda_k)r_{ik} - \beta^{\max}_i(b_{ik})]s_ix_{ik}\rho_i(b_{ik})\\ 
& \textstyle + ~\sum_{k \in \calK} \lambda_k m_k
\end{split}
\end{equation}
Note that the above implies a natural interpretation of the dual variables as related to scaling factors, namely $1 - \lambda_k$ for each $k \in \calK$, to reduce the bid prices based on the fact that each campaign has limited budget. The dual function is defined in the standard way:
\begin{equation}\label{dual_fcn_def}
L^\ast(\lambda) := \max_{\bx \in \calS, \bb \geq 0} \ L(\bx, \bb, \lambda) \ ,
\end{equation}
and we define the dual problem as:
\begin{equation}\label{dual_prob}
\begin{array}{lcl}
\underset{\lambda}{\mathrm{minimize}} & & \displaystyle L^\ast(\lambda) \\
\mathrm{subject \ to} & & 0 \leq \lambda_k \leq 1 \ \ \text{for all } k \in \calK \ .
\end{array}
\end{equation}
Note that the nonnegativity constraints above are standard in Lagrangian duality to ensure that \eqref{dual_prob} provides a valid upper bound on \eqref{poi_deterministic}. The upper bound constraints, i.e., $\lambda_k \leq 1$ for all $k \in \calK$, are without loss of generality since whenever $\lambda_k > 1$ the dual function $L^\ast(\lambda)$ is only improved by instead setting $\lambda_k = 1$. It is always the case that $L^\ast(\cdot)$ is a convex function, but in general it may not be differentiable. Nevertheless, Procedure \ref{algo-subgrad} precisely describes how to compute a \emph{subgradient} of $L^\ast(\cdot)$ at $\lambda$ and is based on Proposition \ref{prop:truth}. Theorem \ref{prop_dual_fcn} demonstrates that Procedure \ref{algo-subgrad} computes valid subgradients and also summarizes the most important properties of the dual function $L^\ast(\cdot)$ and the dual problem \eqref{dual_prob}.

\floatname{algorithm}{Procedure}
\begin{algorithm}
\caption{Computing a Subgradient of $L^\ast(\cdot)$}\label{algo-subgrad}
\begin{algorithmic}
\STATE {\bf Input:} $\lambda \in \bbR^{|\calK|}$ such that $0 \leq \lambda_k \leq 1$ for all $k$ . \\
\STATE 1. For each $(i,k) \in \calE$, set:
\begin{equation*}
\begin{split}
&b^\ast(\lambda)_{ik} \gets (1 - \lambda_k)r_{ik} \ , \text{ and } \\
&\pi(\lambda)_{ik} \gets [b^\ast(\lambda)_{ik} - \beta^{\max}_i(b^\ast(\lambda)_{ik})]s_i\rho_i(b^\ast(\lambda)_{ik})
\end{split}
\end{equation*}
\vspace{-3mm}
\STATE 2. Compute $\bx^\ast(\lambda) \in \arg\max_{\bx \in \calS}\left\{  \sum_{(i,k) \in \calE}\pi(\lambda)_{ik}x_{ik} \right\}$ using the greedy algorithm described in Remark \ref{calScompute}.
\STATE 3. For each $k \in \calK$, set:
\begin{equation*}
\textstyle g(\lambda)_k \gets m_k ~-~ \sum_{i \in \calI_k}r_{ik}s_ix^\ast(\lambda)_{ik}\rho_i(b^\ast(\lambda)_{ik}) \ .
\end{equation*}
\vspace{-3mm}
\STATE {\bf Output:} $g(\lambda) \in \partial L^\ast(\lambda)$ .
\end{algorithmic}
\end{algorithm}

\medskip
\begin{theorem}\label{prop_dual_fcn}
We have the following properties:
\begin{itemize}
\item[\emph{(i)}] $L^\ast(\cdot)$ is finite and convex everywhere on $\bbR^{|\calK|}$.
\item[\emph{(ii)}] For any $\lambda \in \bbR^{|\calK|}$ such that $0 \leq \lambda_k \leq 1$ for all $k$, Procedure \ref{algo-subgrad} computes a valid subgradient $g(\lambda) \in \partial L^\ast(\lambda)$.
\item[\emph{(iii)}] Let $P^\ast$ denote the optimal objective function value of the primal problem \eqref{poi_deterministic}, and let $D^\ast$ denote the optimal objective function value of the dual problem \eqref{dual_prob}. Then, for any $(\bx, \bb)$ that is feasible for \eqref{poi_deterministic} and any $\lambda$ that is feasible for \eqref{dual_prob}, it holds that:
\begin{equation*}
\pi(\bx, \bb) \leq P^\ast \leq D^\ast \leq L^\ast(\lambda) \ .
\end{equation*}
\end{itemize}
\end{theorem}
\begin{proof}
Fix $\lambda \in \bbR^{|\calK|}$ such that $\lambda_k \leq 1$ for all $k$ and consider Proposition \ref{prop:truth} where we replace $r_{ik}$ with the modified value $(1 - \lambda_k)r_{ik}$. It is clear that Proposition \ref{prop:truth} also holds with these modified values. Then, in this case, by \eqref{rearrange} it follows that the subproblem appearing in the definition of $L^\ast(\cdot)$ in \eqref{dual_fcn_def} is exactly the same as \eqref{poi_prop} (except for the term $\sum_{k \in \calK} \lambda_k m_k$, which is just a constant since $\lambda$ is fixed).
Therefore, by Proposition \ref{prop:truth}, it holds that $L^\ast(\lambda)$ is finite. The case when $\lambda_k$ is possibly greater than 1 for some $k$ is requires a simple extension of Proposition \ref{prop:truth} that allows for possibly negative values of $r_{ik}$. Convexity of $L^\ast(\cdot)$ follows since $L^\ast(\cdot)$ is a pointwise maximum of linear functions. 

To prove {\em (ii)}, again by Proposition \ref{prop:truth} it follows that Steps (1.) and (2.) of Procedure \ref{algo-subgrad} are computing a solution of the subproblem in the definition of $L^\ast(\cdot)$ given in \eqref{dual_fcn_def}, i.e., it holds that $(\bx^\ast(\lambda), \bb^\ast(\lambda)) \in \arg\max_{\bx \in \calS, \bb \geq 0}L(\bx, \bb, \lambda)$. By \eqref{original_lag}, Step (3.) of Procedure \ref{algo-subgrad} is computing the partial gradient of $L(\bx^\ast(\lambda), \bb^\ast(\lambda), \lambda)$ (holding $(\bx^\ast(\lambda), \bb^\ast(\lambda))$ fixed), i.e., it holds that $g(\lambda) = \nabla_{\lambda}L(\bx^\ast(\lambda), \bb^\ast(\lambda), \lambda)$. Therefore, for any $\lambda^\prime \in \bbR^{|\calK|}$ it holds that:
\begin{align*}
L^\ast(\lambda^\prime) &\geq L(\bx^\ast(\lambda), \bb^\ast(\lambda), \lambda^\prime) \\
&\geq L(\bx^\ast(\lambda), \bb^\ast(\lambda), \lambda) + g(\lambda)^T(\lambda^\prime - \lambda) \\
&= L^\ast(\lambda) + g(\lambda)^T(\lambda^\prime - \lambda) \ ,
\end{align*}
which by definition implies that $g(\lambda) \in \partial L^\ast(\lambda)$. The first inequality above follows from the definition of $L^\ast(\cdot)$, the second inequality holds since $L(\bx, \bb, \lambda)$ is convex in $\lambda$, and the equality holds by $(\bx^\ast(\lambda), \bb^\ast(\lambda)) \in \arg\max_{\bx \in \calS, \bb \geq 0}L(\bx, \bb, \lambda)$.
Finally, {\em (iii)} is standard in Lagrangian duality and we omit its proof.
\end{proof}

Algorithm \ref{algo-twophase} presents our two phase procedure for obtaining an approximate solution $(\hat\bx, \hat\bb)$ of problem \eqref{poi_deterministic}. In Phase 1, we solve the dual problem \eqref{dual_prob}. As mentioned previously, we suggest the use of simple subgradient methods (see, for example, \cite{nesterov2013introductory} and the references therein), with the use of Procedure \ref{algo-subgrad} to compute subgradients, in order to solve this problem. In our experiments in Section \ref{sec:expts}, we use the basic method of projected subgradient descent with step-sizes proportional to $1/\sqrt{T}$ where $T$ is the iteration counter. In this case, as evident from Procedure \ref{algo-subgrad}, the subgradients will remain bounded and therefore we may apply classical convergence results for this method, which state that the objective function value optimality gap converges to zero at the rate of $O(1/\sqrt{T})$ \cite{nesterov2013introductory}. Moreover, the per iteration cost of this method is dominated by the cost of computing a subgradient, which, as is clear from Procedure \ref{algo-subgrad} and Remark \ref{calScompute}, is $O(|\calE|)$.

\floatname{algorithm}{Algorithm}
\begin{algorithm}
\caption{Two Phase Lagrangian Relaxation-based Scheme for Problem \eqref{poi_deterministic}}\label{algo-twophase}
\begin{algorithmic}
\STATE {\bf Phase 1: Solve Lagrangian Relaxation}
\INDSTATE Solve the dual problem \eqref{dual_prob} to near global optimality 
\INDSTATE using a subgradient method, and return dual variables $\hat\lambda$
\INDSTATE and dual objective value $\hat D \gets L^\ast(\hat \lambda)$.
\STATE {\bf Phase 2: Primal Recovery}
\INDSTATE 1. Set bid prices $\hat{\bb}$: $\hat{b}_{ik} \gets (1 - \hat{\lambda}_k)r_{ik}$ for all $(i,k) \in \calE$.
\INDSTATE 2. Consider the primal problem \eqref{poi_deterministic} with the $\bb$ variables 
\INDSTATE fixed at the values $\hat{\bb}$, and solve the resulting linear 
\INDSTATE optimization problem to obtain allocation probabilities $\hat{\bx}$.
\STATE {\bf Output:} Approximate primal solution $(\hat{\bx}, \hat{\bb})$, primal objective value $\pi(\hat\bx, \hat\bb)$, and dual upper bound $\hat D$.
\end{algorithmic}
\end{algorithm}

In Phase 2 of Algorithm \ref{algo-twophase}, we suggest a heuristic to construct an approximate primal solution $(\hat{\bx}, \hat{\bb})$ based on the previously computed dual solution $\hat \lambda$. First, we use the natural correspondence suggested by \eqref{rearrange} to set the bid prices $\hat{\bb}$. Then, we fix these bid prices and solve the resulting linear optimization problem from \eqref{poi_deterministic} to obtain allocation probabilities $\hat{\bx}$. In our experiments, we use the off-the-shelf solver Gurobi although more sophisticated and scalable approaches may also be employed. Finally, Algorithm \ref{algo-twophase} outputs the approximate primal solution $(\hat{\bx}, \hat{\bb})$ along with the dual objective value $\hat D$. Using item {\em (iii)} of Theorem \ref{prop_dual_fcn}, we may use $\hat D$ to obtain a useful bound on the suboptimality of $(\hat{\bx}, \hat{\bb})$, namely $\pi(\hat{\bx}, \hat{\bb}) \leq P^\ast \leq \hat D$. 

\section{Computational Experiments}\label{sec:expts}
In this section, we present the results of several computational experiments wherein we applied our two-phase solution procedure to synthetic data examples and compared its performance to a baseline policy. Before discussing the experimental results, it is important to clarify how our model, and in particular the output of Algorithm \ref{algo-twophase}, should be applied in a practical, online environment. Policy \ref{algo-online} precisely describes the sequence of events that occur when a new impression of type $i$ arrives to ad exchange and also describes how the decision variables $(\hat{\bx}, \hat{\bb})$ resulting from Algorithm \ref{algo-twophase} would be used to make decisions in real-time. (Recall that the decision variables $\bx$ represent probabilities of selecting campaigns, and that $1 - \sum_{k \in \calK_i} x_{ik}$ represents the probability of not participating in an auction for a new impression of type $i$. In Policy \ref{algo-online} below, the symbol ``0'' is used to encode a ``null campaign'' that represents this option of refraining from participating in the auction.)

\floatname{algorithm}{Policy}
\begin{algorithm}
\caption{Online Policy Implied by Algorithm \ref{algo-twophase}
}\label{algo-online}
\begin{algorithmic}
\STATE {\bf Input:} Approximate primal solution $(\hat{\bx}, \hat{\bb})$ from Algorithm \ref{algo-twophase} and new impression arrival $i \in \calI$.
\\ $ \ $ \\
\STATE 1. Sample a campaign $\tilde{k} \in \calK_i \cup \{0\}$ according to the distribution implied by the values $x_{ik}$ for $k \in \calK_i$ and $1 - \sum_{k \in \calK_i} x_{ik}$. If $\tilde k = 0$ (the sampled campaign is null) or the sampled campaign $\tilde k$ has depleted its budget, then break.
\STATE 2. Enter bid price $\hat{b}_{i\tilde{k}}$. If the auction is won, then pay the ad exchange an amount to equal to the second price. If the auction is not won, then break.
\STATE 3. Show an ad for campaign $\tilde k$. If a click happens, then deduct $q_{\tilde k}$ from the budget of campaign $\tilde k$ and earn revenue $q_{\tilde k}$.
\end{algorithmic}
\end{algorithm}

We also refer to Policy \ref{algo-online} as the ``Lagrangian policy.'' Notice that Policy \ref{algo-online} is particularly conservative in dealing with what happens when a campaign depletes its budget. Indeed, if a campaign with depleted budget is sampled then the DSP does not participate in the auction. Policy \ref{algo-online} may be improved by incorporating the idea of model predictive control whereby Algorithm \ref{algo-twophase} is rerun every time a campaign depletes its budget, or possibly at periodic time intervals. We compare Algorithm \ref{algo-twophase} and correspondingly Policy \ref{algo-online} against a simple ``greedy policy'' that is often employed in practice. Indeed, the greedy policy has the same basic flow of events as Policy \ref{algo-online} with two major differences in how decisions are made:  {\em (i)} at Step (1.) the greedy policy selects, among those campaigns in $\calK_i$ with budgets that are not yet depleted, the campaign $\tilde k$ with the largest eCPI value of $r_{ik}$, and {\em (ii)} at Step (2.) the greedy policy uses $r_{i\tilde{k}}$, the eCPI value of the selected campaign, as the bid price.

\paragraph{Synthetic Data Examples} Let us now describe how the synthetic data examples were generated and how the corresponding simulations were conducted. Throughout this discussion, all relevant random variables are generated independently unless otherwise mentioned. Furthermore, throughout our experiments, it is assumed that the optimization model developed herein is \emph{correctly specified} in that all distributional information used by our model (e.g., in \eqref{poi_deterministic} and Algorithm \ref{algo-twophase}) accurately reflect the corresponding random variables in our simulations. 

Now, to generate our synthetic instances, we first fix the sizes of the sets $\calK$ and $\calI$, and for each campaign $k \in \calK$ (and also for each impression type $i \in \calI$) we generate a ``quality score'' $Q_k$ (resp., $Q_i$) that is uniformly distributed on $[0,1]$. The quality scores are intended to reflect the ``desirability'' of each impression type and each campaign and are used to generate the main parameters of our model. Indeed, for each impression type $i \in \calI$, the set $\calK_i$ is constructed by sampling edges independently with probability $Q_i$. Hence $|\calK_i| \sim \text{Bin}(|\calK|, Q_i)$, where $\text{Bin}(n, p)$ denotes a binomial random variable with $n$ trials and success probability parameter $p$. Moreover, for each impression type $i \in \calI$, $B^{\max}_i$ is taken to be the maximum of $\text{Bin}(M, Q_i)$ independent random variables that are uniformly distributed on $[0,1]$, where $M$ is an integer parameter dictating the ``size of the market.'' The click-through-rate values $\theta_{ik}$ are defined by $\theta_{ik} := Q_i \cdot Q_k$.

In all of our experiments, we set $|\calK| = 100$, $M = 10$, $s_i = 5000$ for all $i \in \calI$, and the CPC value $q_k = 1$ for all $k \in \calK$. In our first experiment, we generated a single problem instance, referred to as Example A, that additionally had $|\calI| = 100$ and the budget parameters set to $m_k = 50$ for all $k \in \calK$. In this case, as verified by the dual upper bound $\hat D$, Algorithm \ref{algo-twophase} was able to solve problem \eqref{poi_deterministic} to within 13\% of optimality. We compared the Policy \ref{algo-online} implied by Algorithm \ref{algo-twophase} to the ``greedy policy'' described earlier by simulating the impression arrival, real-time bidding, and click processes. In our simulations, we assume that impressions arrive to the ad exchange according to $|\calI|$ independent Poisson processes that are ``merged'' together. This assumption implies that $S_i$ is a Poisson random variable, and the time horizon $T$ of the overall arrival process is set so that $s_i = 5000$ for each $i \in \calI$.

\paragraph{Results} The top left table in Figure \ref{thefig} presents the main results of our first experiment on Example A. We ran 500 simulation runs comparing our Lagrangian relaxation approach, i.e., Policy \ref{algo-online} to the greedy baseline policy. Each policy saw the same exact sequence of impression arrivals and the same sequence of realized $B^{\max}_i$ values during each individual simulation run. For each simulation run, we computed the relative profit, relative cost, and relative revenue of the two policies, where each relative statistic is computed as the Lagrangian statistic relative to the greedy baseline, e.g., $\text{Relative Profit} := \frac{\text{Lagrangian Profit}}{\text{Greedy Profit}}$. The results in the tables are averaged over 500 such simulation runs. As the top left table demonstrates, the Lagrangian policy is able to achieve significantly higher profit levels and lower costs than the greedy baseline. Interestingly, the Lagrangian policy also achieves lower revenue levels than the greedy policy. This makes good intuitive sense since the Lagrangian policy uses bid prices that are shaded down by a factor of $1 - \lambda_k$ as compared to the greedy policy, and moreover the Lagrangian policy should make smarter allocation decisions whereby, under the Lagrangian policy, a particularly valuable campaign would wait for better opportunities before depleting its budget as compared to the greedy policy. The bottom left table in Figure \ref{thefig} reports the budget utilization (defined as total revenue divided by $\sum_{k \in \calK}m_k$) and profit margin (defined as total profit divided by total revenue) statistics for each policy, which confirms our intuition and the results presented in top left table. Example B constitutes our second experiment, whereby we took the same exact problem instance and made one modification, namely instead of using constant budgets across the different campaigns we allowed the budget to be correlated with the quality score of each campaign so that $m_k := 50Q_k$. As Figure \ref{thefig} demonstrates, the profit improvement of the Lagrangian policy over the greedy baseline is even more dramatic in this case.

\begin{figure}[h!]
\vspace{-2mm}
\begin{multicols}{2}
\scalebox{0.68}{
\begin{tabular}{c}
\begin{tabular}{ccc}
\hline
{\bf (Lag./Gr.)}	   & {\bf Example A} & {\bf Example B} \\ \hline
{\bf Relative Profit}  & {\Large 1.257}     & {\Large 1.576}     \\
{\bf Relative Cost}    & {\Large 0.286}     & {\Large 0.431}     \\
{\bf Relative Revenue} & {\Large 0.759}     & {\Large 0.677}     \\ \hline
\end{tabular}
\\ \\ \\
\begin{tabular}{ccccc}
                   & \multicolumn{2}{c}{{\bf Example A}} & \multicolumn{2}{c}{{\bf Example B}} \\ \hline
                   & {\bf Lag.}      & {\bf Gr.}      & {\bf Lag.}      & {\bf Gr.}      \\ \hline
{\bf Budget Util.} & {\Large 0.483}           & {\Large 0.636}      & {\Large 0.542}           & {\Large 0.801}      \\
{\bf Profit/Revenue}     & {\Large 0.807}           & {\Large 0.487}       & {\Large 0.500}          & {\Large 0.215}       \\ \hline
\end{tabular}
\end{tabular}
}
\vfill\null
\columnbreak
\includegraphics[scale = 0.2]{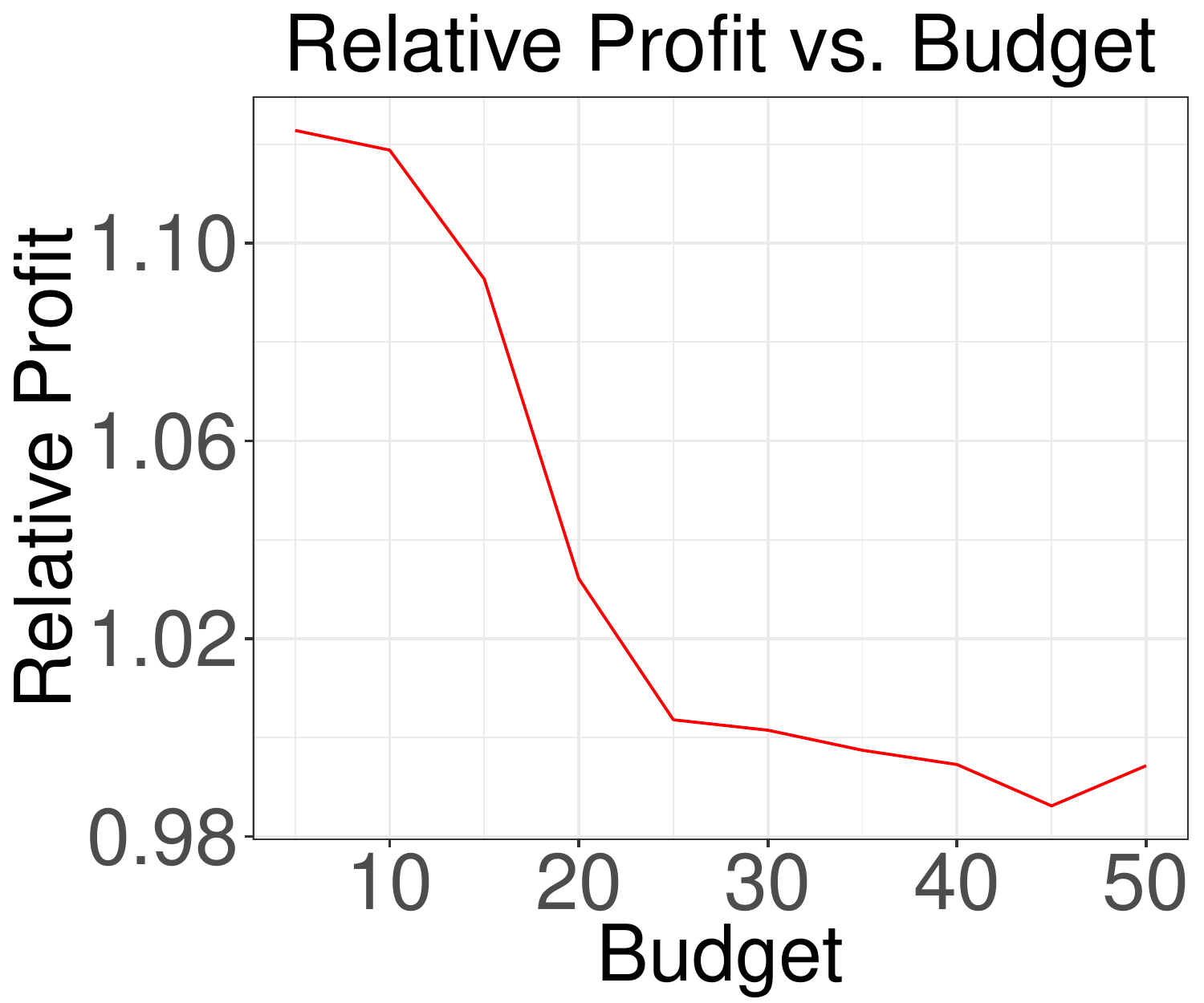}
\end{multicols}
\vspace{-8mm}
\caption{Lagrangian two-phase policy vs. greedy policy}
\label{thefig}
\end{figure}

In our third experiment, we generated a fresh problem instance using all of the same parameters as before except for two differences:  {\em (i)} we set $|\calI| = 10$, and {\em (ii)} we varied the budget parameters $m_k \in \{5,10,15,20,25,30,35,40,45,50\}$ (also we took $m_k$ to be constant across the different campaigns) and reran Algorithm \ref{algo-twophase} for each of the 10 budget values. The right side of Figure \ref{thefig} plots the relative profit of the two policies versus these budget values, and for each of the 10 budget values the relative profit statistic was averaged over 500 simulation runs. Clearly, the Lagrangian policy has significantly larger profits over the greedy policy for small budget values, but as the budget increases this improvement is diminished. This makes sense since the Lagrangian policy is based on accounting for the budget constraints in \eqref{poi_deterministic} via dual variables. Indeed, as the budget values become larger, the budget constraints are less active and Proposition \ref{prop:truth} implies that Policy \ref{algo-online} and the greedy policy are exactly the same when $m_k$ is large enough for each $k \in \calK$.

Let us conclude this section by mentioning a few directions for future research. First, it would be very valuable to also perform some computational experiments comparing the Lagrangian policy to the greedy policy using a real advertising dataset. Second, it would be very interesting to extend our methodology, in particular problem \eqref{poi_deterministic} and Algorithm \ref{algo-twophase} to different pricing models, such as the CPM pricing model, with performance constraints. Finally, it would be interesting to examine the benefits of more sophisticated stochastic or robust optimization approaches that more carefully account for the uncertainty in the impression arrivals and the real-time bidding environment. The authors plan to pursue all of these directions in future research.

\bibliographystyle{ACM-Reference-Format}
\bibliography{reference} 

\end{document}